\documentclass[11pt,a4paper]{article}

\usepackage{amsmath}
\usepackage{amsthm}
\usepackage{amssymb}
\usepackage{amscd}
\usepackage[all]{xy}

\title{Strongly Liftable Schemes and the Kawamata-Viehweg 
Vanishing in Positive Characteristic
\footnote{This paper was partially supported by the National Natural Science 
Foundation of China (Grant No.\ 10901037) and Ph.D.\ Programs Foundation of 
Ministry of Education of China (Grant No.\ 20090071120004).}}
\author{Qihong Xie}
\date{}
\theoremstyle{plain}
\newtheorem{prop}{Proposition}[section]
\newtheorem{lem}[prop]{Lemma}
\newtheorem{thm}[prop]{Theorem}
\newtheorem{cor}[prop]{Corollary}

\theoremstyle{definition}
\newtheorem{defn}[prop]{Definition}
\newtheorem*{ack}{Acknowledgments}
\newtheorem*{nota}{Notation}

\theoremstyle{remark}
\newtheorem{rem}[prop]{Remark}

\newcommand{\Q}{\mathbb Q}

\newcommand{\Z}{\mathbb Z}

\newcommand{\F}{\mathbb F}
\newcommand{\A}{\mathbb A}
\newcommand{\PP}{\mathbb P}
\newcommand{\OO}{\mathcal O}
\newcommand{\II}{\mathcal I}
\newcommand{\EE}{\mathcal E}

\newcommand{\MM}{\mathcal M}
\newcommand{\HH}{\mathcal H}
\newcommand{\LL}{\mathcal L}
\newcommand{\GG}{\mathcal G}

\newcommand{\TT}{\mathcal T}

\newcommand{\Pic}{\mathop{\rm Pic}\nolimits}

\newcommand{\Supp}{\mathop{\rm Supp}\nolimits}
\newcommand{\Exc}{\mathop{\rm Exc}\nolimits}
\newcommand{\ch}{\mathop{\rm char}\nolimits}
\newcommand{\Hom}{\mathop{\rm Hom}\nolimits}
\newcommand{\Spec}{\mathop{\rm Spec}\nolimits}
\newcommand{\Proj}{\mathop{\bf Proj}\nolimits}
\newcommand{\proj}{\mathop{\rm Proj}\nolimits}

\newcommand{\ext}{\mathop{\rm Ext}\nolimits}
\newcommand{\divisor}{\mathop{\rm div}\nolimits}
\newcommand{\ra}{\rightarrow}

\newcommand{\wt}{\widetilde}

\begin{document}
\maketitle

\begin{abstract}
A smooth scheme $X$ over a field $k$ of positive characteristic 
is said to be strongly liftable, if $X$ and all prime divisors 
on $X$ can be lifted simultaneously over $W_2(k)$. 
In this paper, we give some concrete examples and properties of 
strongly liftable schemes. As an application, we prove that the 
Kawamata-Viehweg vanishing theorem in positive characteristic 
holds on any normal projective surface which is birational to 
a strongly liftable surface. 
\end{abstract}

\setcounter{section}{0}
\section{Introduction}\label{S1}

As is well known, the Kawamata-Viehweg vanishing theorem plays a crucial 
role in birational geometry of algebraic varieties, and it is of several 
forms, where the most general form is stated for log pairs which have 
only Kawamata log terminal singularities \cite[Theorem 1-2-5]{kmm}.

\begin{thm}[Kawamata-Viehweg vanishing]\label{1.1}
Let $X$ be a normal projective variety over an algebraically closed 
field $k$ with $\ch(k)=0$, $B=\sum b_iB_i$ an effective $\Q$-divisor 
on $X$, and $D$ a $\Q$-Cartier Weil divisor on $X$. 
Assume that $(X,B)$ is Kawamata log terminal (KLT for short), 
and $D-(K_X+B)$ is ample. Then $H^i(X,D)=0$ holds for any $i>0$.
\end{thm}

In what follows, we always work over {\it an algebraically closed 
field $k$ of characteristic $p>0$} unless otherwise stated. 
The Kawamata-Viehweg vanishing theorem for smooth projective varieties 
in positive characteristic has first been proved by Hara \cite{hara} 
under the lifting condition over $W_2(k)$ of certain log pairs.

\begin{thm}[Kawamata-Viehweg vanishing in char.\ $p>0$]\label{1.2}
Let $X$ be a smooth projective variety over $k$ of dimension $d$, 
$H$ an ample $\Q$-divisor on $X$, and $D$ a simple normal crossing 
divisor containing $\Supp(\langle H\rangle)$. 
Assume that $(X,D)$ admits a lifting over $W_2(k)$. Then 
\[ H^i(X,\Omega_X^j(\log D)(-\ulcorner H\urcorner))=0 \,\,\,\,
\hbox{holds for any} \,\,\,\, i+j<\inf(d,p). \]
In particular, $H^i(X,K_X+\ulcorner H\urcorner)=0$ holds 
for any $i>d-\inf(d,p)$.
\end{thm}

To obtain a practical version of the Kawamata-Viehweg vanishing 
theorem in positive characteristic instead of Theorem \ref{1.2}, 
we have to take the lifting property of log pairs into account. 
For this purpose, we introduce the following notion. 
A smooth scheme $X$ is said to be strongly liftable, if $X$ and 
all prime divisors on $X$ can be lifted simultaneously over $W_2(k)$ 
(see Definition \ref{2.3} for more details).

First of all, we can give many concrete examples of strongly liftable schemes.

\begin{thm}\label{1.3}
The following schemes are strongly liftable:
\begin{itemize}
\item[(i)] $\A^n_k$, $\PP^n_k$ and a smooth projective curve;

\item[(ii)] a smooth projective variety of Picard number 1 
which is a complete intersection in $\PP^n_k$;

\item[(iii)] a smooth projective rational surface;

\item[(iv)] a smooth projective surface whose relatively minimal 
model is $\PP_C(\LL_1\oplus\LL_2)$, where $C$ is a smooth 
projective curve and $\LL_i$ are invertible sheaves on $C$.
\end{itemize}
\end{thm}

A direct consequence of Theorem \ref{1.2} is that the Kawamata-Viehweg 
vanishing theorem holds on strongly liftable schemes. Furthermore, we 
have a stronger result in dimension two.

\begin{thm}\label{1.4}
The Kawamata-Viehweg vanishing theorem holds on any normal projective 
surface which is birational to a strongly liftable surface.
\end{thm}

Theorems \ref{1.3} and \ref{1.4} imply the following consequences, 
where Corollary \ref{1.5} is the main theorem of \cite[Theorem 1.4]{xie09}.

\begin{cor}\label{1.5}
The Kawamata-Viehweg vanishing theorem holds on rational surfaces.
\end{cor}

\begin{cor}\label{1.6}
Let $Y$ be a smooth projective variety of dimension $d\geq 3$ 
which is a complete intersection in $\PP^n_k$, $f:X\ra Y$ a 
composition of blow-ups along closed points, and $H$ an ample 
$\Q$-divisor on $X$ with $\Supp(\langle H\rangle)$ being simple 
normal crossing. Then $H^i(X,K_X+\ulcorner H\urcorner)=0$ 
holds for any $i>d-\inf(d,p)$.
\end{cor}

In \S \ref{S2}, we will give various definitions of liftings 
over $W_2(k)$ and some preliminary results. In \S \ref{S3}, 
we will give some examples and properties of strongly liftable 
schemes. \S \ref{S4} is devoted to the proof of the main results. 
For the necessary notions and results in birational geometry, 
e.g.\ Kawamata log terminal singularity, we refer the reader to 
\cite{kmm} and \cite{km}.

\begin{nota}
We use $\sim$ to denote linear equivalence, 
$\equiv$ to denote numerical equivalence, 
and $[B]=\sum [b_i] B_i$ (resp.\ 
$\ulcorner B\urcorner=\sum \ulcorner b_i\urcorner B_i$, 
$\langle B\rangle=\sum \langle b_i\rangle B_i$, 
$\{B\}=\sum\{b_i\}B_i$) to denote the round-down (resp.\ 
round-up, fractional part, upper fractional part) 
of a $\Q$-divisor $B=\sum b_iB_i$, where for a real number $b$, 
$[b]:=\max\{ n\in\Z \,|\,n\leq b \}$, $\ulcorner b\urcorner:=-[-b]$, 
$\langle b\rangle:=b-[b]$ and $\{ b\}:=\ulcorner b\urcorner-b$.
\end{nota}

\begin{ack}
I would like to express my gratitude to Professors Luc Illusie 
and Frans Oort for useful comments.
\end{ack}

\section{Preliminaries}\label{S2}

Let us first recall some definitions from \cite[Definition 8.11]{ev}.

\begin{defn}\label{2.1}
Let $W_2(k)$ be the ring of Witt vectors of length two of $k$. 
Then $W_2(k)$ is flat over $\Z/p^2\Z$, and $W_2(k)\otimes_{\Z/p^2\Z}\F_p=k$. 
For the explicit construction and further properties of $W_2(k)$, 
we refer the reader to \cite[II.6]{se}. The following definition 
generalizes the definition \cite[1.6]{di} of liftings of $k$-schemes 
over $W_2(k)$.

Let $X$ be a noetherian scheme over $k$, and $D=\sum D_i$ a reduced Cartier 
divisor on $X$. A lifting of $(X,D)$ over $W_2(k)$ consists of a scheme 
$\wt{X}$ and closed subschemes $\wt{D_i}\subset\wt{X}$, all defined and 
flat over $W_2(k)$ such that $X=\wt{X}\times_{\Spec W_2(k)}\Spec k$ and 
$D_i=\wt{D_i}\times_{\Spec W_2(k)}\Spec k$. We write 
$\wt{D}=\sum \wt{D_i}$ and say that $(\wt{X},\wt{D})$ is a lifting 
of $(X,D)$ over $W_2(k)$, if no confusion is likely.
\end{defn}

In the above definition, assume further that $X$ is smooth 
over $k$ and $D=\sum D_i$ is simple normal crossing. 
If $(\wt{X},\wt{D})$ is a lifting of $(X,D)$ over $W_2(k)$, 
then $\wt{X}$ is smooth over $W_2(k)$ and $\wt{D}=\sum \wt{D_i}$ 
is relatively simple normal crossing over $W_2(k)$, i.e.\ $\wt{X}$ 
is covered by affine open subsets $\{U_\alpha\}$, such that each 
$U_\alpha$ is \'{e}tale over $\A^n_{W_2(k)}$ via coordinates 
$\{ x_1,\cdots,x_n \}$ and $\wt{D}|_{U_\alpha}$ is defined by 
the equation $x_1\cdots x_\nu=0$ with $1\leq\nu\leq n$ 
(see \cite[Lemmas 8.13, 8.14]{ev}).

If $\wt{X}$ is a lifting of $X$ over $W_2(k)$, then there is an exact 
sequence of $\OO_{\wt{X}}$-modules:
\begin{eqnarray}
0\ra \OO_X\stackrel{p}{\ra} \OO_{\wt{X}}\stackrel{r}{\ra} 
\OO_X\ra 0, \label{es1}
\end{eqnarray}
where $p(x):=px$ and $r(\wt{x}):=\wt{x}\mod p$ 
for $x\in\OO_X,\wt{x}\in\OO_{\wt{X}}$ (see \cite[Lemma 8.13]{ev}).

\begin{defn}\label{2.2}
Let $X$ be a smooth scheme over $k$, $D=\sum D_i$ a reduced divisor on 
$X$, and $Z$ a closed subscheme of $X$ smooth over $k$ of codimension 
$s\geq 2$. A mixed lifting of $(X,D+Z)$ over $W_2(k)$ consists of a smooth 
scheme $\wt{X}$ over $W_2(k)$, closed subschemes $\wt{D_i}\subset\wt{X}$ 
flat over $W_2(k)$, and a closed subscheme $\wt{Z}\subset\wt{X}$ smooth 
over $W_2(k)$ such that $X=\wt{X}\times_{\Spec W_2(k)}\Spec k$, 
$D_i=\wt{D_i}\times_{\Spec W_2(k)}\Spec k$ and 
$Z=\wt{Z}\times_{\Spec W_2(k)}\Spec k$. 
We write $\wt{D}=\sum \wt{D_i}$ and say that $(\wt{X},\wt{D}+\wt{Z})$ 
is a mixed lifting of $(X,D+Z)$ over $W_2(k)$, if no confusion is likely.
\end{defn}

In the above definition, either $D=\emptyset$ or $Z=\emptyset$ is allowed. 
Obviously, if $Z=\emptyset$ then a mixed lifting $(\wt{X},\wt{D})$ of 
$(X,D)$ is indeed a lifting of $(X,D)$ over $W_2(k)$.

\begin{defn}\label{2.3}
Let $X$ be a smooth scheme over $k$. $X$ is said to be strongly 
liftable over $W_2(k)$, if the following two conditions hold:

(i) $X$ is liftable over $W_2(k)$;

(ii) there is a lifting $\wt{X}$ of $X$, such that for any prime 
divisor $D$ on $X$, $(X,D)$ has a lifting $(\wt{X},\wt{D})$ over 
$W_2(k)$ in the sense of Definition \ref{2.1}.
\end{defn}

Therefore, for a strongly liftable smooth scheme $X$, both $X$ and any 
effective divisor on $X$ can be lifted simultaneously over $W_2(k)$.

\begin{lem}\label{2.4}
Let $X$ be a smooth scheme over $k$, $D$ a reduced divisor on $X$, 
and $Z\subset X$ a closed subscheme smooth over $k$ of codimension 
$s\geq 2$. Let $\pi:X'\ra X$ be the blow-up of $X$ along $Z$ with the 
exceptional divisor $E$, $D'=\pi_*^{-1}D$ the strict transform of $D$. 
Assume that $(X,D+Z)$ admits a mixed lifting over $W_2(k)$. 
Then $(X',D'+E)$ admits a mixed lifting over $W_2(k)$.
\end{lem}

\begin{proof}
Let $(\wt{X},\wt{D}+\wt{Z})$ be a mixed lifting of $(X,D+Z)$ over $W_2(k)$. 
Then $\wt{Z}\subset\wt{X}$ is a closed subscheme smooth over $W_2(k)$ of 
codimension $s\geq 2$. Let $\wt{I}$ be the ideal sheaf of $\wt{Z}$ in 
$\wt{X}$, $\wt{\pi}:\wt{X}'\ra\wt{X}$ the blow-up of $\wt{X}$ along $\wt{Z}$ 
with the exceptional divisor $\wt{E}$, and $\wt{D}'=\wt{\pi}_*^{-1}\wt{D}$. 
By \cite[Corollary II.7.15]{ha}, we have the following commutative 
diagram:
\[
\xymatrix{
X'' \ar[d]_{\pi'} \ar@{^{(}->}[r] & \wt{X}' \ar[d]^{\wt{\pi}} \\
X \ar@{^{(}->}[r] & \wt{X} 
}
\]
where $\pi':X''\ra X$ is the blow-up of $X$ with respect to the ideal 
sheaf $\wt{I}\otimes_{W_2(k)}k=I$, the ideal sheaf of $Z$ in $X$. 
Hence $X''=X'$ and $\pi'=\pi$. Since $\wt{X}$ is smooth over $W_2(k)$, 
so is $\wt{X}'$. Note that $\wt{X}'\times_{\Spec W_2(k)}
\Spec k=\Proj(\oplus_i \wt{I}^i)\times_{\Spec W_2(k)}\Spec k
=\Proj(\oplus_i \wt{I}^i\otimes_{W_2(k)}k)=\Proj(\oplus_i I^i)=X'$,
so $\wt{X}'$ is a lifting of $X'$ over $W_2(k)$. It is easy to see that 
$\wt{D}'\times_{\Spec W_2(k)}\Spec k=D'$ and 
$\wt{E}\times_{\Spec W_2(k)}\Spec k=E$, hence 
$(X',D'+E)$ has a mixed lifting $(\wt{X}',\wt{D}'+\wt{E})$ over $W_2(k)$.
\end{proof}

\begin{lem}\label{2.5}
Let $X$ be a smooth scheme over $k$, and $P\in X$ a closed point. 
If $X$ has a lifting over $W_2(k)$, then $(X,P)$ has a mixed lifting 
over $W_2(k)$.
\end{lem}

\begin{proof}
Let $\wt{X}$ be a lifting of $X$ over $W_2(k)$, 
$X\hookrightarrow \wt{X}$ the induced closed immersion, 
and $\eta: \wt{X}\ra \Spec W_2(k)$ the structure morphism. 
Let $\Spec k\hookrightarrow X$ be the closed immersion 
associated to the closed point $P\in X$, and 
$\Spec k\hookrightarrow\Spec W_2(k)$ the natural closed 
immersion. We have the following commutative square:
\[
\xymatrix{
\Spec k \ar@{^{(}->}[d] \ar@{^{(}->}[r] & X \ar@{^{(}->}[r] & 
\wt{X} \ar[d]^{\eta} \\
\Spec W_2(k) \ar@{=}[rr] \ar@{-->}[urr]^{\xi} & & \Spec W_2(k)
}
\]

Since $\Spec k\hookrightarrow\Spec W_2(k)$ is a closed immersion 
with ideal sheaf square zero and $\eta: \wt{X}\ra \Spec W_2(k)$ is 
smooth, there is a morphism $\xi: \Spec W_2(k)\ra \wt{X}$ such that 
the induced diagrams are commutative. Since $\xi$ is a section of 
$\eta$, it defines a closed subscheme $\wt{P}\subset\wt{X}$ smooth 
over $W_2(k)$. It follows from the upper commutativity that 
$P=\wt{P}\times_{\Spec W_2(k)}\Spec k$ holds. Therefore, 
$(\wt{X},\wt{P})$ is a mixed lifting of $(X,P)$ over $W_2(k)$.
\end{proof}

\begin{prop}\label{2.6}
Let $X$ be a smooth scheme over $k$ with $\dim X\geq 2$, 
$P\in X$ a closed point, and $\pi:X'\ra X$ the blow-up of $X$ 
along $P$. If $X$ is strongly liftable (resp.\ liftable) over $W_2(k)$, 
then $X'$ is strongly liftable (resp.\ liftable) over $W_2(k)$.
\end{prop}

\begin{proof}
For liftability, it follows from Lemmas \ref{2.5} and \ref{2.4}. 
For strong liftability, let $\wt{X}$ be a lifting of $X$ satisfying 
the condition (ii) in Definition \ref{2.3}. Let $\wt{P}\in\wt{X}$ 
be a mixed lifting of $P\in X$ as in Lemma \ref{2.5}, and 
$\pi:\wt{X}'\ra\wt{X}$ the blow-up of $\wt{X}$ along $\wt{P}$. 
A similar argument to the proof of Lemma \ref{2.4} shows that 
$\wt{X}'$ is a lifting of $X'$. Therefore, it suffices to show 
that $\wt{X}'$ satisfies the condition (ii) in Definition \ref{2.3}. 
For any prime divisor $D'$ on $X'$, either $\pi_*(D')$ is a prime 
divisor or $\pi_*(D')=P$. If $\pi_*(D')=D$ is a prime divisor, 
then $(X,D+P)$ has a mixed lifting $(\wt{X},\wt{D}+\wt{P})$ 
by assumption, hence $(X',D')$ has a lifting $(\wt{X}',\wt{D}')$ 
by Lemma \ref{2.4}. If $\pi_*(D')=P$, i.e. $D'$ is the exceptional 
divisor $E$ of $\pi$, then via the mixed lifting $(\wt{X},\wt{P})$ 
of $(X,P)$, $(X',E)$ has a lifting $(\wt{X}',\wt{E})$ by Lemma \ref{2.4}.
\end{proof}

\section{Examples of strongly liftable schemes}\label{S3}

\begin{lem}\label{3.1}
$\A^n_k$ and $\PP^n_k$ are strongly liftable over $W_2(k)$.
\end{lem}

\begin{proof}
Any prime divisor $D$ on $\A^n_k$ (resp.\ $\PP^n_k$) is defined by 
the equation $f=0$, where $f$ is a polynomial in $k[x_1,\cdots,x_n]$ 
(resp.\ a homogeneous polynomial in $k[x_0,\cdots,x_n]$). Therefore, 
we can lift each coefficient of $f$ over $W_2(k)$ to obtain a polynomial 
$\wt{f}$ in $W_2(k)[x_1,\cdots,x_n]$ (resp.\ a homogeneous polynomial 
$\wt{f}$ in $W_2(k)[x_0,\cdots,x_n]$). Then the divisor $\wt{D}$ defined 
by $\wt{f}=0$ is a lifting of $D$ over $W_2(k)$.
\end{proof}

\begin{lem}\label{3.2}
Any smooth projective curve is strongly liftable over $W_2(k)$.
\end{lem}

\begin{proof}
It follows from \cite[Proposition 2.12]{il} that for any 
smooth scheme $X$ over $k$, there is an obstruction 
$o(X)\in\ext^2_{\OO_X}(\Omega^1_X,\OO_X)=H^2(X,\TT_X)$ to 
the liftability of $X$ over $W_2(k)$, i.e.\ $o(X)=0$ 
if and only if $X$ is liftable over $W_2(k)$. 
Any smooth projective curve $C$ has a lifting $\wt{C}$ over 
$W_2(k)$ since $H^2(C,\TT_C)=0$. Fix such a lifting $\wt{C}$, 
then for any closed point $P\in C$, $(C,P)$ has a lifting 
$(\wt{C},\wt{P})$ by Lemma \ref{2.5}.
\end{proof}

\begin{lem}\label{3.3}
Let $X$ be a smooth projective variety of Picard number 1 which is a 
complete intersection in $\PP^n_k$. Then $X$ is strongly liftable over 
$W_2(k)$.
\end{lem}

\begin{proof}
In fact, if $X\subset\PP^n_k$ is a complete intersection with 
$\dim X\geq 3$, then the Picard group $\Pic(X)$ is the free abelian 
group generated by the isomorphism class of the invertible sheaf 
$\OO_X(1)$, which implies automatically the Picard number $\rho(X)=1$ 
(see \cite[XII, Corollaire 3.7]{sga2}). If $\dim X=2$, then 
\cite[XI, Th\'eor\`eme 1.8]{sga7} says that $\Pic(X)$ is torsion-free 
and $\OO_X(1)$ is not divisible in $\Pic(X)$, hence $\Pic(X)$ is the 
free abelian group generated by the isomorphism class of $\OO_X(1)$ 
since $\rho(X)=1$.

Let $\dim X=d$, and $f_1,\cdots,f_{n-d}$ homogeneous polynomials in 
$k[x_0,\cdots,x_n]$ such that 
$X=\proj k[x_0,\cdots,x_n]/(f_1,\cdots,f_{n-d})$. Let $D$ be an 
irreducible divisor on $X$. Then $\OO_X(D)\cong\OO_X(r)$ with 
$r\geq 1$. First of all, by lifting each coefficient of $f_i$, 
we can take $\wt{f_1},\cdots,\wt{f}_{n-d}\in W_2(k)[x_0,\cdots,x_n]$ 
lifting $f_1,\cdots,f_{n-d}$ respectively. The scheme $\wt{X}$ defined 
by $\proj W_2(k)[x_0,\cdots,x_n]/(\wt{f}_1,\cdots,\wt{f}_{n-d})$ is 
a lifting of $X$.

Since $X$ is a smooth subscheme of $\PP^n$, the sequence 
$\{f_1,\cdots,f_{n-d}\}$ is regular, which gives rise to 
the following Koszul resolution of the sheaf $\OO_X$:
\begin{eqnarray}
0\ra \wedge^{n-d}\MM\ra \wedge^{n-d-1}\MM\ra \cdots\ra \wedge^1\MM\ra 
\OO_{\PP^n}\ra \OO_X\ra 0, \label{es2}
\end{eqnarray}
where $\MM=\OO_{\PP^n}^{\oplus n-d}$. The exact sequence (\ref{es2}) 
factorizes into short exact sequences:
\begin{eqnarray}
0\ra \II_{i}\ra \wedge^{i-1}\MM\ra \II_{i-1}\ra 0, \label{es3}
\end{eqnarray}
where $\II_0=\OO_X$, $\II_{n-d}=\wedge^{n-d}\MM$ and 
$\II_{1},\cdots,\II_{n-d-1}$ are defined by (\ref{es2}). 
We will use frequently the vanishing $H^j(\PP^n,\wedge^i\MM(r))=0$ 
for any $i\geq 0$, $j>0$ and $r\geq 1$, which follows from the 
vanishing $H^j(\PP^n,\OO_{\PP^n}(r))=0$ for any $j>0$ and $r\geq 1$. 
To prove that $H^0(\PP^n,\OO_{\PP^n}(r))\ra H^0(X,\OO_X(D))$ is 
surjective, it suffices to show $H^1(\PP^n,\II_1(r))=0$. 
By using the exact sequence (\ref{es3}) and the above vanishing, 
the induction reduces to show $H^{n-d}(\PP^n,\II_{n-d}(r))=0$, 
which follows from the above vanishing again.

The surjectivity of $H^0(\PP^n,\OO_{\PP^n}(r))\ra H^0(X,\OO_X(D))$ 
implies that we can take a hypersurface $H\subset\PP^n$ of degree $r$ 
such that $D=X\cap H$. Take a lifting $\wt{H}$ of $H$ as in Lemma 
\ref{3.1}, and define $\wt{D}=\wt{X}\cap\wt{H}$. Then it is easy to 
see that $\wt{D}\subset\wt{X}$ is a lifting of $D\subset X$.
\end{proof}

Let $X$ be a smooth projective scheme over $k$, $\wt{X}$ a lifting of $X$ 
over $W_2(k)$, and $\iota: X\hookrightarrow\wt{X}$ the closed immersion. 
We have a natural group homomorphism $\iota^*:\Pic(\wt{X})\ra\Pic(X)$ 
defined by $\wt{\LL}\mapsto\iota^*\wt{\LL}$. 
There is an exact sequence of abelian sheaves:
\begin{eqnarray}
0\ra \OO_X\stackrel{q}{\ra} \OO_{\wt{X}}^*\stackrel{r}{\ra} 
\OO_X^*\ra 1, \label{es4}
\end{eqnarray}
where $q(x):=1+px$, $r(\wt{x}):=\wt{x}\mod p$ 
for $x\in\OO_X,\wt{x}\in\OO_{\wt{X}}$, 
which gives rise to an exact sequence: 
\begin{eqnarray}
H^1(\wt{X},\OO_{\wt{X}}^*)\stackrel{\pi}{\ra} H^1(X,\OO_X^*)
\ra H^2(X,\OO_X). \label{es5}
\end{eqnarray}
We can identify $\iota^*:\Pic(\wt{X})\ra\Pic(X)$ with 
$\pi:H^1(\wt{X},\OO_{\wt{X}}^*)\ra H^1(X,\OO_X^*)$, by using the 
canonical isomorphisms $H^1(X,\OO_X^*)\stackrel{\sim}{\ra}\Pic(X)$ 
and $H^1(\wt{X},\OO_{\wt{X}}^*)\stackrel{\sim}{\ra}\Pic(\wt{X})$. 
Let $D$ be a prime divisor on $X$, and $\LL_D=\OO_X(D)$ 
the associated invertible sheaf on $X$. 
Assume that $\iota^*:\Pic(\wt{X})\ra\Pic(X)$ is surjective. 
Then $\LL_D$ has a lifting $\wt{\LL}_D$ on $\wt{X}$. Furthermore, 
by tensoring the exact sequence (\ref{es1}) with $\wt{\LL}_D$, 
we have a natural exact sequence of sheaves of $\OO_{\wt{X}}$-modules:
\begin{eqnarray}
0\ra \LL_D\stackrel{p}{\ra} \wt{\LL}_D\stackrel{r}{\ra} 
\LL_D\ra 0, \label{es6}
\end{eqnarray}
which gives rise to an exact sequence:
\begin{eqnarray}
H^0(\wt{X},\wt{\LL}_D)\stackrel{\pi_D}{\longrightarrow} H^0(X,\LL_D)\ra 
H^1(X,\LL_D). \label{es7}
\end{eqnarray}

\begin{prop}\label{3.4}
Let $X$ be a smooth projective scheme over $k$, and $\wt{X}$ a lifting 
of $X$ over $W_2(k)$. We use the same notations as above. Then $X$ is 
strongly liftable if and only if the following two conditions hold:

(i) $\pi:H^1(\wt{X},\OO_{\wt{X}}^*)\ra H^1(X,\OO_X^*)$ is surjective;

(ii) $\pi_D:H^0(\wt{X},\wt{\LL}_D)\ra H^0(X,\LL_D)$ is surjective 
for any prime divisor $D$ on $X$.
\end{prop}

\begin{proof}
Sufficiency: For any prime divisor $D$ on $X$, $D$ corresponds to 
a nonzero section $s\in H^0(X,\LL_D)$ with $D=\divisor_0(s)$. 
Since $\pi$ is surjective, we can take a lifting $\wt{\LL}_D$ of 
$\LL_D$ and consider $\pi_D:H^0(\wt{X},\wt{\LL}_D)\ra H^0(X,\LL_D)$. 
Since $\pi_D$ is surjective, there is a nonzero section 
$\wt{s}\in H^0(\wt{X},\wt{\LL}_D)$. Let $\wt{D}=\divisor_0(\wt{s})$. 
Then it is easy to see that $(\wt{X},\wt{D})$ is a lifting of 
$(X,D)$ over $W_2(k)$.

Necessity: Since $X$ is strongly liftable, any effective divisor 
$D\subset X$ has a lifting $\wt{D}\subset\wt{X}$. For any invertible 
sheaf $\LL$ on $X$, write $\LL=\OO_X(D_1-D_2)$ with $D_i$ effective. 
Lift $D_i$ to $\wt{D}_i$ for $i=1,2$. Then the invertible sheaf 
$\wt{\LL}:=\OO_{\wt{X}}(\wt{D}_1-\wt{D}_2)$ lifts $\LL$, which 
proves the condition (i). The condition (ii) is obvious.
\end{proof}

\begin{cor}\label{3.5}
Let $X$ be a smooth projective scheme over $k$, and $\wt{X}$ a lifting 
of $X$ over $W_2(k)$. We use the same notations as above. Then $X$ is 
strongly liftable if the following two conditions hold:

(i) $H^2(X,\OO_X)=0$;

(ii) for any prime divisor $D$ on $X$, either 
$\pi_D:H^0(\wt{X},\wt{\LL}_D)\ra H^0(X,\LL_D)$ is surjective 
or $H^1(X,\OO_X(D))=0$.
\end{cor}

\begin{proof}
From the exact sequence (\ref{es5}), it follows that if 
$H^2(X,\OO_X)=0$ then $\pi:H^1(\wt{X},\OO_{\wt{X}}^*)\ra 
H^1(X,\OO_X^*)$ is surjective. From the exact sequence 
(\ref{es7}), it follows that if $H^1(X,\OO_X(D))=0$ then 
$\pi_D:H^0(\wt{X},\wt{\LL}_D)\ra H^0(X,\LL_D)$ is surjective.
\end{proof}

\begin{lem}\label{3.6}
Let $S$ be a scheme over $k$, and $\wt{S}$ a lifting of $S$ over $W_2(k)$. 
Let $X$ be a smooth $S$-scheme, and $D\subset X$ a divisor which is 
relatively simple normal crossing over $S$. Then there is an obstruction 
$o(X/S,D)\in\ext_{\OO_X}^2(\Omega_{X/S}^1(\log D),\OO_X)=
H^2(X,\TT_{X/S}(-\log D))$ to the liftability of $(X,D)$ over $\wt{S}$, 
i.e.\ $o(X/S,D)=0$ if and only if $(X,D)$ is liftable over $\wt{S}$.
\end{lem}

\begin{proof}
The case when $D=\emptyset$ was verified directly in 
\cite[Proposition 2.12]{il}. For the general case, 
\cite[Proposition 8.22]{ev} just showed that the isomorphisms of 
liftings of $(X,D)$ over $\wt{S}$ form a torsor under 
the group $\Hom_{\OO_X}(\Omega_{X/S}^1(\log D),\OO_X)$, 
hence by a similar argument to that of \cite[Proposition 2.12]{il}, 
we get the required obstruction 
$o(X/S,D)\in\ext_{\OO_X}^2(\Omega_{X/S}^1(\log D),\OO_X)=
H^2(X,\TT_{X/S}(-\log D))$. 
\end{proof}

\begin{lem}\label{3.7}
Let $C$ be a smooth projective curve, and $\EE$ a locally free sheaf 
on $C$ of rank 2. Let $X=\PP(\EE)$, $f:X\ra C$ the natural projection, 
and $E$ a section of $f$. Then $f:(X,E)\ra C$ has a lifting 
$\wt{f}:(\wt{X},\wt{E})\ra\wt{C}$ over $W_2(k)$.
\end{lem}

\begin{proof}
Let $\wt{C}$ be a lifting of $C$. By Lemma \ref{3.6}, the obstruction 
to the liftability of $(X,E)$ over $\wt{C}$ is $o(X/C,E)\in\ext^2_{\OO_X}
(\Omega^1_{X/C}(\log E),\OO_X)=H^2(X,\TT_{X/C}(-\log E))$. By Serre duality, 
$H^2(X,\TT_{X/C}(-\log E))^\vee=H^0(X,\Omega^1_{X/C}(\log E)\otimes\omega_X)
=H^0(X,\omega_{X/C}\otimes\omega_X\otimes\OO_X(E))$, which vanishes 
by an easy calculation.
\end{proof}

\begin{prop}\label{3.8}
Let $C$ be a smooth projective curve, $\GG$ an invertible sheaf on $C$, 
and $\EE=\OO_C\oplus\GG$. Let $X=\PP(\EE)$, and $f:X\ra C$ the natural 
projection. Then $X$ is strongly liftable over $W_2(k)$.
\end{prop}

\begin{proof}
Let $E$ be a section of $f$ with $\OO_X(E)\cong\OO_X(1)$. 
By Lemma \ref{3.7}, we can take and fix such a lifting 
$\wt{f}:(\wt{X},\wt{E})\ra\wt{C}$ of $f:(X,E)\ra C$. 
Since $H^2(X,\OO_X)=0$, by Corollary \ref{3.5}, to prove 
the strong liftability of $X$, it suffices to show that 
for any irreducible curve $D$ on $X$, 
$\pi_D:H^0(\wt{X},\wt{\LL}_D)\ra H^0(X,\LL_D)$ is surjective, 
where $\LL_D=\OO_X(D)$.

Write $D\sim aE+f^*H$, where $a\geq 0$ and $H$ is a divisor on $C$. 
Let $\wt{H}$ be a divisor on $\wt{C}$ lifting $H$, $\HH=\OO_X(H)$ and 
$\wt{\HH}=\OO_{\wt{X}}(\wt{H})$. Then $\LL_D=\OO_X(D)$ has a lifting 
$\wt{\LL}_D=\OO_{\wt{X}}(a\wt{E}+\wt{f}^*\wt{H})$. Note that
\begin{eqnarray}
H^0(X,\LL_D)=H^0(C,f_*\OO_X(aE+f^*H))=H^0(C,S^a(\EE)\otimes\HH)
=\oplus_{i=0}^a H^0(C,\GG^i\otimes\HH), \nonumber \\
H^0(\wt{X},\wt{\LL}_D)=H^0(\wt{C},\wt{f}_*\OO_{\wt{X}}(a\wt{E}+
\wt{f}^*\wt{H}))=H^0(\wt{C},S^a(\wt{\EE})\otimes\wt{\HH})
=\oplus_{i=0}^a H^0(\wt{C},\wt{\GG}^i\otimes\wt{\HH}). \nonumber
\end{eqnarray}
Therefore, $\pi_D:H^0(\wt{X},\wt{\LL}_D)\ra H^0(X,\LL_D)$ factorizes into 
$\pi_D^i:H^0(\wt{C},\wt{\GG}^i\otimes\wt{\HH})\ra H^0(C,\GG^i\otimes\HH)$. 
Since $C$ is strongly liftable by Lemma \ref{3.2}, 
$\pi_D^i$ are surjective for all $i$, hence $\pi_D$ is surjective.
\end{proof}

\begin{rem}\label{3.9}
We can generalize Proposition \ref{3.8} as follows by an analogous 
proof. Let $Y$ be a strongly liftable smooth projective variety, 
$\EE$ a decomposable locally free sheaf on $Y$, and $X=\PP(\EE)$. 
Then $X$ is strongly liftable over $W_2(k)$.
\end{rem}

\begin{rem}\label{3.10}
Let $C$ be a smooth projective curve, $\EE$ a locally free sheaf 
on $C$ of rank 2, and $X=\PP(\EE)$. If $\EE$ is indecomposable, 
then $X$ is not necessarily strongly liftable. Such an example 
has been given in \cite[Theorem 3.1]{xie07}. More precisely, 
if $C$ is a Tango curve, then there is an indecomposable 
locally free sheaf $\EE$ on $C$ of rank 2, such that the pull-back 
$F^*\EE$ by the Frobenius $F:C\ra C$ is decomposable, which gives 
rise to an irreducible curve $C'$ on $X=\PP(\EE)$, such that $(X,C')$ 
cannot be lifted over $W_2(k)$. It is also an example of liftable 
but not strongly liftable scheme such that the condition (ii) in 
Proposition \ref{3.4} is not satisfied.
\end{rem}

\begin{rem}\label{3.11}
An abelian variety is another example of liftable but not strongly 
liftable scheme such that the condition (i) in Proposition \ref{3.4} 
is not satisfied. We proceed the following argument provided by 
Oort. Let $(X,\LL)$ be a polarized abelian variety of dimension $g$. 
Then the universal deformation space of $X$, say for $\Spec(R)$, 
is smooth over $W(k)$ of relative dimension $g^2$, where $W(k)$ is 
the ring of Witt vectors of $k$, and the universal deformation space 
of $(X,\LL)$, say for $\Spec(R/I)$, is smooth over $W(k)$ of relative 
dimension $<g^2$ (see \cite{oo71,oo79} for more details). Therefore, 
we can take a ring homomorphism $\varphi:R\ra W_2(k)$ which induces the 
identity map on residue closed fields and satisfies $\varphi(I)\neq 0$. 
The existence of such $\varphi$ means that $X$ can be lifted over $W_2(k)$, 
while $(X,\LL)$ cannot be lifted over $W_2(k)$ at the same time.
\end{rem}

\section{Proof of the main results}\label{S4}

\begin{proof}[Proof of Theorem \ref{1.3}]
(i) follows from Lemmas \ref{3.1} and \ref{3.2}. 
(ii) follows from Lemma \ref{3.3}. 
Both (iii) and (iv) follow from Propositions \ref{3.8} and \ref{2.6} 
since any locally free sheaf on $\PP^1$ is decomposable.
\end{proof}

The following vanishing result \cite[Corollary 2.2.5]{kk} is useful, 
which holds in arbitrary characteristic.

\begin{lem}\label{4.1}
Let $h: Y\ra X$ be a proper birational morphism between normal surfaces 
with $Y$ smooth and with exceptional locus $E=\cup_{i=1}^s E_i$. 
Let $L$ be an integral divisor on $Y$, $0\leq b_1,\cdots,b_s<1$ 
rational numbers, and $N$ an $h$-nef $\Q$-divisor on $Y$. Assume 
$ L\equiv K_Y+\sum_{i=1}^s b_iE_i+N$. Then $R^1h_*\OO_Y(L)=0$ holds.
\end{lem}

We shall prove Theorem \ref{1.4} in the following explicit form.

\begin{thm}\label{4.2}
Let $X$ be a normal projective surface, 
$D$ a $\Q$-Cartier Weil divisor on $X$, and 
$B$ an effective $\Q$-divisor such that $(X,B)$ is KLT 
and that $D-(K_X+B)$ is ample. If $X$ is birational to a strongly 
liftable smooth projective surface $Z$, then $H^1(X,D)=0$ holds.
\end{thm}

\begin{proof}
Take a log resolution $h: Y\ra X$ such that the following three 
conditions hold:
\[
\xymatrix{
  & \ar[dl]_h Y \ar[dr]^{\pi} & \\
X \ar@{-->}[rr] & & Z
}
\]

(i) $Y$ is a smooth projective surface over $k$, and we 
can write $K_Y+h_*^{-1}B \equiv h^*(K_X+B)+\sum_ia_iE_i$, 
where $E_i$ are the exceptional curves of $h$ and $a_i>-1$ 
for all $i$;

(ii) $G=\Supp(h_*^{-1}B)\cup\Exc(h)$ is simple normal crossing;

(iii) $\pi:Y\ra Z$ is a birational morphism.

Let $D_Y=\ulcorner h^*D+\sum_ia_iE_i\urcorner$. Since 
$\ulcorner \sum_ia_iE_i\urcorner\geq 0$ is supported by $\Exc(h)$, 
we have $h_*\OO_Y(D_Y)=\OO_X(D)$ by the projection formula. 
Since $\{h^*D+\sum_ia_iE_i\}$ is supported by $\Exc(h)$, 
we can take $0<\delta_i\ll 1$ such that

(i) $\bigl{[}h_*^{-1}B+\{h^*D+\sum_ia_iE_i\}+
\sum_i\delta_iE_i\bigr{]}=0$.

(ii) $D_Y-(K_Y+h_*^{-1}B+\{h^*D+\sum_ia_iE_i\}+
\sum_i\delta_iE_i) \equiv h^*(D-(K_X+B))-\sum_i\delta_iE_i$ is ample.

Let $B_Y=h_*^{-1}B+\{h^*D+\sum_ia_iE_i\}+\sum_i\delta_iE_i$. 
Then $H_Y=D_Y-(K_Y+B_Y)$ is ample, $\Supp(\langle H_Y\rangle)=\Supp(B_Y)$ 
is simple normal crossing, and $K_Y+\ulcorner H_Y\urcorner=D_Y$. Note that
\[ D_Y \equiv K_Y+\{h^*D+\sum_ia_iE_i\}+h^*(D-(K_X+B))
+h_*^{-1}B. \]
By Lemma \ref{4.1}, we have $R^1h_*\OO_Y(D_Y)=0$, hence 
$H^1(Y,D_Y)=H^1(X,h_*\OO_Y(D_Y))=H^1(X,D)$. 

Since $\pi:Y\ra Z$ is a birational morphism between smooth projective 
surfaces, it is a composition of blow-ups \cite[Corollary V.5.4]{ha}. 
Since $Z$ is strongly liftable over $W_2(k)$, so is $Y$ by Proposition 
\ref{2.6}, hence $(Y,G)$ admits a lifting over $W_2(k)$. 
Since $G$ contains $\Supp(\langle H_Y\rangle)$, we have 
$H^1(X,D)=H^1(Y,D_Y)=H^1(Y,K_Y+\ulcorner H_Y\urcorner)=0$ 
by Theorem \ref{1.2}.
\end{proof}

\begin{rem}\label{4.3}
It seems impossible to generalize Theorem \ref{4.2} (i.e.\ 
Theorem \ref{1.4}) to the higher dimensional cases because of 
the following reasons: 
First of all, it is unknown whether the strong liftability is 
stable under the blow-ups along higher dimensional centers, 
while the center in Proposition \ref{2.6} is of dimension zero. 
Secondly, the analogue of Lemma \ref{4.1} fails definitely on 
higher dimensional varieties. 
Finally, the higher dimensional minimal model program is more 
involved, and extremal contractions and flips are more general 
and complicated than blow-ups or blow-downs.
\end{rem}

\begin{proof}[Proof of Corollary \ref{1.6}]
It follows from Lemma \ref{3.3}, Proposition \ref{2.6} 
and Theorem \ref{1.2}.
\end{proof}

\textsc{School of Mathematical Sciences, Fudan University, 
Shanghai 200433, China}

\textit{E-mail address}: \texttt{qhxie@fudan.edu.cn}

\end{document}